\documentclass[10pt]{article}

\usepackage[utf8]{inputenc}
\usepackage{mathtools,mathrsfs,amssymb,amsthm,amsmath,bm}
\usepackage{comment}
\usepackage{hyperref}
\usepackage{url}
\usepackage[a4paper, left = 1in, right = 1in, top =1in, bottom = 1in]{geometry} 
\usepackage{dsfont}   
\usepackage[british]{babel}

\newtheorem{thm}{Theorem}

\newtheorem{cor}{Corollary}
\newtheorem{pro}{Proposition}

\theoremstyle{definition}
\newtheorem*{rem}{Remark}

\makeatletter
\renewenvironment{proof}[1][\proofname] {\par\pushQED{\qed}\normalfont\topsep6\p@\@plus6\p@\relax\trivlist\item[\hskip\labelsep\bfseries#1\@addpunct{.}]\ignorespaces}{\popQED\endtrivlist\@endpefalse}
\makeatother

\def\[#1\]{\begin{align*}#1\end{align*}}

\newcommand{\R}{\mathbb{R}}
\newcommand{\N}{\mathbb{N}}
\newcommand{\ceq}{\coloneqq}
\newcommand{\eqc}{\eqqcolon}
\newcommand{\eps}{\varepsilon}
\def\B{\mathscr{B}}
\newcommand{\ltx}[1]{\ \text{#1}}
\newcommand{\df}{\mathop{}\!\mathrm{d}}
\newcommand{\M}
{\mathbb{M}}
\newcommand{\st}{\,\vert\,}
\newcommand{\pfqed}{\tag*{\qedhere}}

\begin{document}
\title{On Almost Uniform Continuity of Borel Functions on Polish Metric Spaces}
\author{Yu-Lin Chou\thanks{Yu-Lin Chou, Institute of Statistics, National Tsing Hua University, Hsinchu 30013, Taiwan,  R.O.C.; Email: \protect\url{y.l.chou@gapp.nthu.edu.tw}.}}
\date{}
\maketitle

\begin{abstract}
We show that, on any given finite Borel measure space with the ambient
space being a Polish metric space, every Borel real-valued function
is almost a bounded, uniformly continuous function in the sense that
for every $\eps > 0$ there is some bounded, uniformly continuous
function such that the set of points at which they would not agree
has measure $< \eps$. In particular, this result complements the
known result of almost uniform continuity of Borel real-valued functions
on a finite Radon measure space whose ambient space is a locally compact
metric space. As direct applications in connection with some common
modes of convergence, under our assumptions it holds that i) for every
Borel real-valued function there is some sequence of bounded, uniformly
continuous functions converging in measure to it, and ii) for every
 bounded, Borel real-valued function there is some sequence of bounded, uniformly
continuous functions converging in $L^{p}$ to it.\\

\end{abstract}

\section{Introduction}

Let $\Omega$ be a metric space; let $\M$ be a finite Borel measure
over $\Omega$. It follows from a well-known version of Lusin's Theorem
(e.g. Theorem 2.24 in Rudin \cite{r}) that, if $\Omega$ is locally
compact, if $\M$ is Radon, and if $f: \Omega \to \R$ is Borel-measurable,
then for every $\eps > 0$ there is some bounded, uniformly continuous
function $\Omega \to \R$ such that the set of points at which they
possibly disagree has measure $< \eps$. For our purposes, we refer
to a Borel function $\Omega \to \R$ satisfying the conclusion of
the above proposition as ($\M$-)\textit{almost uniformly continuous},
where $\Omega$ is simply a metric space and $\M$, with respect to
which the meaning of ``almost'' is clearly assigned, is simply a finite
Borel measure over $\Omega$. Thus boundedness is also a requirement
of almost uniform continuity.

Now, depending on the purposes, local compactness is not always a
helping property. For instance, there are metric spaces that are interesting
and important in analysis but that are not locally compact; it can
be shown that the space $\R^{\N}$ of real sequences, equipped with
the product metric (in terms of summation) of the (equivalent) Euclidean
metric $(x,y) \mapsto 1 \wedge |x-y|$ of $\R$, is not locally compact,
and that the classical Wiener space $C([0,1], \R)$ is not locally
compact with respect to the uniform metric; no closed ball in either
space is compact. On the other hand, these spaces are indeed complete
and separable with respect to the respective metrics, i.e. they are
Polish metric spaces. Given the importance of Polish metric spaces,
in particular of compact metric spaces, in analysis (and geometry),
it would be desirable to have an almost-uniform-continuity result,
serving Polish metric spaces and requiring no local compactness, for
Borel real-valued functions, such that the assumption on the underlying
finite measure is hopefully mild. 

It turns out that it suffices for the underlying (finite) measure
to be Borel. In this short article, we obtain the desired result that
\textit{every Borel real-valued function on a Polish metric space is almost uniformly continuous with respect to a pre-specified finite Borel measure}.
Indeed, given an arbitrary metric space and a finite Borel measure
over it, one can already assert that every Borel real-valued function
is almost a continuous function in the following sense: It follows
from a proposition in Federer \cite{f} (Section 2.3.6) that every
Borel real-valued function on the metric space has the property that
for every $\eps > 0$ there is some continuous function such that
the set of points where they possibly disagree has measure $< \eps$.
The proposition is obtained directly from the Federer's version\footnote{Although the Federer's proof is for outer measures in the ``usual'' sense, it happens to apply to Borel measures in the ``usual'' sense. A measure in the geometric-measure-theoretic sense is precisely an outer measure in the ``usual'' sense.}
(Section 2.3.5, the proof being applicable) of Lusin's theorem and
the Tietze's extension theorem. However, without further assumptions,
the approximating continuous functions need not be bounded or uniformly
continuous. 

The main message of the proof of our main result is, rather than claiming
a ``significant'' advancement in the classical topics, that one might
as well obtain the readily applicable result —{\,\,}almost uniform
continuity of Borel real-valued functions on a Polish metric space
taken as a finite Borel measure space —{\,\,}simply with new twists
of known facts. Some interesting by-products of the main result, under
the same assumptions, are i) every Borel real-valued function is the
(essential) convergence-in-measure limit of bounded, uniformly continuous
functions, and ii) every  bounded, Borel real-valued function
is the (essential) $L^{p}$-limit of bounded, uniformly continuous
functions. After introducing the necessary preliminaries, we proceed
to the proofs.

\section{Preliminaries}

By a \textit{Borel} measure over a metric space $\Omega$ we mean
a measure defined on the Borel sigma-algebra $\B_{\Omega}$
of $\Omega$ generated by the topology of $\Omega$ induced by the
given metric of $\Omega$.

We will stick to the standard measure-theoretic definitions of outer regularity
and inner regularity of measures. For our purposes, it would be convenient
to introduce another kind of regularity associated with measures.
If $\M$ is a finite Borel measure over a metric space $\Omega$, the measure
$\M$ is called \textit{co-outer regular}\footnote{In the literature of probability theory, co-outer regularity is sometimes also termed inner regularity, and is associated with Borel probability measures over a metric space. Since our arguments will involve both inner regularity in the standard measure-theoretic sense and co-outer regularity, and since every Borel probability measure over a metric space is a finite Borel measure, we choose to assign a new name to the property.}
if and only if $\M(B) = \sup\{ \M(F) \mid F \subset B \ltx{is closed}\}$
for every $B \in \B_{\Omega}$. The terminology reflects the elementary
fact that a closed set is the complement of some open set. Since $\sup \{ \M(K) \mid K \subset B \ltx{is compact}\} \leq \sup \{ \M(F) \mid F \subset B \ltx{is closed}\}$
whenever $B \in \B_{\Omega}$, the inner regularity (resp. co-outer
regularity) of a finite Borel measure need not imply co-outer regularity
(resp. inner regularity).

The topology of a metric space always refers to the topology induced
by the given metric. By a \textit{Polish metric space} we mean a
metric space that is complete and separable with respect to the given
metric. If $\Omega$ is a metric space, and if $\M$ is a finite
Borel measure over $\Omega$, we will denote by $L^{0}(\M)$ the collection
of all Borel functions $\Omega \to \R$, 
by $L^{0}_{b}(\M)$ the collection of all bounded, Borel functions $\Omega \to \R$,
by $C(\Omega)$ the collection
of all continuous functions $\Omega \to \R$, 
by $C_{u}(\Omega)$
the collection of all uniformly continuous functions $\Omega \to \R$,
and by $C_{b,u}(\Omega)$ the collection of all bounded, uniformly
continuous functions $\Omega \to \R$. 

Throughout, we will in general write a set $\{ x \in \Omega \mid P(x) \ltx{holds} \}$
obtained by specification simply as $\{ P \}$ whenever no confusion
is possible. Thus, if $f,g: \Omega \to \R$, then $\{ x \in \Omega \mid f(x) \neq g(x) \} = \{ f \neq g \}$.
Moreover, when written in juxtaposition with a measure, the set $\{ f \neq g \}$
will also be written simply as $(f \neq g)$. For example, we have
$\M(\{ f \neq g \}) = \M(f \neq g)$. This notation is common in probability
theory.

If $A_{1}, A_{2}$ are subsets of a topological space, then $A_{1}$
is said to be \textit{relatively dense in} $A_{2}$ if and only if
the closure of $A_{1}$ includes $A_{2}$. If $A_{2}$ coincides with
the given ambient space, then the relative denseness of $A_{1}$ in
$A_{2}$ is simply the denseness of $A_{1}$ in $A_{2}$ in the usual
sense.

We will argue in terms of the language of topology, which may be more
conceptually ``compact”. If $\Omega$ is a metric space,
and if $\M$ is a finite Borel measure over $\Omega$, let \[
V(f,\eps) \ceq \{ g \in L^{0}(\M) \st \M(f \neq g) < \eps \}
\]for every $f \in L^{0}(\M)$ and every $\eps > 0$. 
We have $f \in V(f,\eps)$ for every $f \in L^{0}(\M)$ and every $\eps > 0$; moreover, the triangle inequality ensures that the intersection of any two $V(f,\eps)$ is some union of $V(f,\eps)$. Topologize $L^{0}(\M)$ in terms of the topology
generated by $\{ V(f,\eps) \}_{f,\eps}$.
Considering the conclusion of the aforementioned Rudin's version\footnote{When we say ``Rudin's version'' or ``Federer's version'', we merely make a nominal distinction, which facilitates the communication.}
of Lusin's theorem, we will refer to the topology of $L^{0}(\M)$
thus obtained as a \textit{Lusin topology} (so $\{ V(f,\eps) \}_{f,\eps}$
may naturally be called a \textit{Lusin basis}), for ease of reference.
Accordingly, the topological properties of sets such as closedness
with respect to the Lusin topology of $L^{0}(\M)$ will be referred
to in terms of the modifier ``Lusin''; for instance, a closed subset
of $L^{0}(\M)$ with respect to the corresponding Lusin topology will
also be said to be \textit{Lusin-closed}. Now the almost uniform
continuity of elements of $L^{0}(\M)$ may be translated as follows:
An element $f$ of $L^{0}(\M)$ is almost uniformly continuous if
and only if $f$ lies in the Lusin-closure of $C_{b,u}(\Omega)$.
Moreover, the almost uniform continuity of elements of $L^{0}(\M)$,
which is the conclusion of our main result, may now be stated neatly
as: the space $C_{b,u}(\Omega)$ is Lusin-dense in $L^{0}(\M)$.

\section{Results}

We should like to prove our main result:

\begin{thm}

If $\Omega$ is a Polish metric space, and if $\M$ is a finite Borel
measure over $\Omega$, then $C_{b,u}(\Omega)$ is Lusin-dense in
$L^{0}(\M)$.

\end{thm}

\begin{proof}

We have $C_{b,u}(\Omega) \subset L^{0}(\M)$; so the Lusin-closure
of $C_{b,u}(\Omega)$ is included in $L^{0}(\M)$. It then suffices
to show that every element of $L^{0}(\M)$ lies in the Lusin-closure
of $C_{b,u}(\Omega)$. 

Let $f \in L^{0}(\M)$; let $\eps > 0$. Since every finite Borel
measure over a metric space is both outer regular and co-outer regular,
which is known and may be obtained neatly from an immediate, apparent
generalization of the simple proof of Theorem 1.1. in Billingsley
\cite{b}, adapting the proof of the Federer's version of Lusin's
theorem for arbitrary metric spaces (i.e. Section 2.3.5, Federer \cite{f})
in the apparent way ensures the existence of some closed subset $F_{\eps}$
of $\Omega$ such that $\M(F_{\eps}^{c}) < \eps/2$ and $f|_{F_{\eps}}$
is continuous. 

On the other hand, we claim that $\M$ is in fact also inner regular.
Indeed, for $\M$ to be inner regular it is sufficient for $\M$ to
be inner regular at $\Omega$, i.e. for it to hold that $\M(\Omega) = \sup \{ \M(K) \mid K \subset \Omega \ltx{is compact} \}$.
To see this, fix any $B \in \B_{\Omega}$. Then, since $\M$ is co-outer
regular, for every $\delta > 0$ there is some closed $F \subset B$
such that $\M(F) > \M(B) - \delta/2$; moreover, there is some compact
$K \subset \Omega$ such that $\M(K) > \M(\Omega) - \delta/2$. Then
$K \cap F$ is compact and included in $B$, and \[
\M(K \cap F) 
&= \M(F) - \M(F\setminus K)\\ 
&> \M(B) - \delta/2 - \M(K^{c})\\
&> \M(B) - \delta;  
\]the inner regularity of $\M$ follows. But $\M$ is indeed inner regular
at $\Omega$; this follows from a direct apparent application of the
simple proof of Theorem 1.3 in Billingsley \cite{b}. We have proved
the claim of the inner regularity of $\M$.

Now there is some compact $K \subset F_{\eps}$ such that $\M(F_{\eps}\setminus K) < \eps/2$,
and so $f|_{K}$ is bounded and uniformly continuous. Then the McShane's
extension theorem (Corollary 2, McShane \cite{m}) asserts the existence
of some $g \in C_{b,u}(\Omega)$ such that $g|_{K} = f|_{K}$ (and $g$ preserves the bounds).  Since
$\{ f \neq g \} \subset K^{c}$, we have \[
\M(f \neq g)
&\leq \M(K^{c})\\
&\leq \M(F_{\eps}\setminus K) + \M(F_{\eps}^{c})\\
&< \eps;
\]so $f$ lies in the Lusin-closure of $C_{b,u}(\Omega)$. \end{proof}

\begin{rem}

As some branches of probability theory admitting extensive literature
such as weak convergence theory (e.g. Billingsley \cite{b}) or optimal
transport (e.g. Villani \cite{v}) serve as a natural, significant
context directly deeply connected with Polish metric spaces taken
as a finite Borel measure space, we would stress that the applicability
of Theorem 1 covers the Borel probability spaces whose ambient space
is a Polish metric space, although this remark is technically apparent.
\qed

\end{rem}

There is an interesting application of Theorem 1 for convergence in
measure:

\begin{thm}

If $\Omega$ is a Polish metric space, and if $\M$ is a finite Borel
measure over $\Omega$, then $C_{b,u}(\Omega)$ is dense in $L^{0}(\M)$
with respect to the convergence-in-measure topology of $L^{0}(\M)$.

\end{thm}

\begin{proof}

In accordance with the topological flavor, we topologize $L^{0}(\M)$
in terms of the topology generated by the subsets \[
\{ g \in L^{0}(\M) \st \inf \{r \in [0, +\infty] \st \M(|f - g| \geq r) < r \} < \eps \} 
\]where $f \in L^{0}(\M)$ and $\eps > 0$.
It is well-known (and readily seen) that a sequence
in $L^{0}(\M)$ converges in measure in $L^{0}(\M)$ if and only if
it converges with respect to this topology in $L^{0}(\M)$. Call the
topology the convergence-in-measure topology of $L^{0}(\M)$.

Let $f \in L^{0}(\M)$. If $\eps > 0$, then there is by Theorem
1 some $g \in C_{b,u}(\Omega)$ such that $\M(|f-g| > 0) < \eps/2$,
and so $\M(|f-g| \geq \eps/2) \leq \M(|f-g| > 0) < \eps/2$. This
implies that $g$ lies in the basic (convergence-in-measure-)neighborhood
of $f$ with radius $\eps$. \end{proof}

We have,
as a side observation   potentially of interest,
more information on the relations between the two topologies of $L^{0}(\M)$:
\begin{pro}
If $\Omega$ is a metric space, and if $\M$ is a finite Borel measure over $\Omega$, then every element of the convergence-in-measure topology of    $L^{0}(\M)$ is some union of elements of the Lusin topology of $L^{0}(\M)$.
\end{pro}
\begin{proof}
For convenience,
denote by $d_{c}$ the  (pseudo-)metric defining a basic open set for the convergence-in-measure topology. 

Let $f \in L^{0}(\M)$; let $\eps > 0$; let $d_{c}(f,g) < \eps$. If $h$ is contained in the basic open set $V(g, \frac{\eps - d_{c}(f,g)}{2})$ of the Lusin topology,
then $d_{c}(h,g) \leq \frac{\eps - d_{c}(f,g)}{2} < \eps - d_{c}(f,g)$;
and so
\[
d_{c}(h,f)
\leq d_{c}(h,g) + d_{c}(g,f) 
< \eps. \pfqed
\]
\end{proof}

Theorem 1 may also be applied to obtain an interesting result regarding
$L^{p}$-convergence:

\begin{thm}

If $\Omega$ is a Polish metric space, and if $\M$ is a finite Borel
measure over $\Omega$, then $C_{b,u}(\Omega)$ is relatively $L^{p}$-dense
in $L^{0}_{b}(\M)$ for every $1 \leq p < +\infty$.

\end{thm}

\begin{proof}

Let $1 \leq p < +\infty$. Since $\M$ is by assumption finite, we
have $C_{b,u}(\Omega) \subset L^{0}_{b}(\M) \subset L^{p}(\M)$.
The proof is complete if the $L^{p}$-closure of $C_{b,u}(\Omega)$
includes $L^{0}_{b}(\M)$. 

Let $f \in L^{0}_{b}(\M)$; let $T > 0$ be a bound of $f$.  
We have $f \in L^{0}(\M)$ by definition, and,
given any $\eps > 0$, there is by Theorem 1 some $g \in C_{b,u}(\Omega)$
with $T$ being a bound  such that \[
\M(|f-g| > 0) < \bigg( \frac{\eps}{T} \bigg)^{p} \eqc \eta.
\]Then Minkowski inequality implies\[
\bigg( \int_{\Omega} |f-g|^{p} \df \M \bigg)^{1/p}
&= \bigg( \int_{\{ |f-g| > 0 \}}|f-g|^{p} \df \M \bigg)^{1/p}\\
&< T\eta^{1/p}\\
& = \eps. \pfqed
\] 

\end{proof}

Applying the proof ideas of Theorems 2, 3 and using the Rudin's version
of Lusin's theorem together give 

\begin{pro}

If $\Omega$ is a locally compact metric space, and if $\M$ is a
finite Radon measure over $\Omega$, then i) $C_{b,u}(\Omega)$ is
dense in $L^{0}(\M)$ with respect to the convergence-in-measure topology
of $L^{0}(\M)$, and ii) $C_{b,u}(\Omega)$ is relatively $L^{p}$-dense
in $L^{0}_{b}(\M)$ for every $1 \leq p < +\infty$.\qed

\end{pro}

\begin{rem}

Under the assumptions of Proposition 1, Proposition 1 contains, as
far as almost uniform continuity is concerned, information in addition
to the corollary to Theorem 2.24, i.e. the Rudin's version of Lusin's
theorem, in Rudin \cite{r}; in some directions, Proposition 1 contains
more information.\qed

\end{rem}

As a compact metric space is both locally compact and Polish, a generic
corollary certainly follows:

\begin{cor}

If $\Omega$ is a compact metric space, and if $\M$ is a finite Borel
measure over $\Omega$, then i) $C_{b,u}(\Omega)$ is Lusin-dense
in $L^{0}(\M)$, ii) $C_{b,u}(\Omega)$ is dense in $L^{0}(\M)$ with
respect to the convergence-in-measure topology of $L^{0}(\M)$, and
iii) $C_{b,u}(\Omega)$ is relatively $L^{p}$-dense in $L^{0}_{b}(\M)$
for every $1 \leq p < +\infty$ \qed

\end{cor}

\end{document}